\documentclass[12pt]{article}
\usepackage{amssymb}
\usepackage{amsmath,amsthm}
\usepackage[latin1]{inputenc}
\usepackage{color}
\usepackage{graphicx}
 \newtheorem{remark}{Remark}

 \newtheorem{lemma}[remark]{Lemma}
 \newtheorem{theorem}[remark]{Theorem}
 \newtheorem{proposition}[remark]{Proposition}
 \newtheorem{corollary}[remark]{Corollary}

\title{Roman domination in Cartesian product graphs and strong product graphs}

\author{ Ismael G. Yero$^{1}$ and Juan A.
Rodr\'{\i}guez-Vel\'{a}zquez$^{2}$\\
    \\
$^1${\small Departamento de Matem\'aticas, Escuela Polit\'ecnica Superior de Algeciras}\\
{\small Universidad de C\'adiz,} {\small
Av. Ram\'on Puyol, s/n, 11202 Algeciras, Spain.} \\ {\small
ismael.gonzalez\@@uca.es}\\
$^2${\small Departament d'Enginyeria Inform\`atica i Matem\`atiques}\\
{\small Universitat Rovira i Virgili,}  {\small Av. Pa\"{\i}sos
Catalans 26, 43007 Tarragona, Spain.} \\{\small
juanalberto.rodriguez\@@urv.cat}
\\
}

\begin{document}

\maketitle
\begin{abstract}
A set $S$ of vertices of a graph $G$ is a dominating set for $G$ if
every vertex outside of $S$ is adjacent to at least one vertex
belonging to $S$. The minimum cardinality of a dominating set for $G$ is called the
domination number of $G$. A map $f : V \rightarrow \{0, 1, 2\}$ is a Roman dominating function on a
graph $G$ if for every vertex $v$ with $f(v) = 0$, there exists a
vertex $u$, adjacent to $v$, such that $f(u) = 2$. The weight of a Roman
dominating function is given by $f(V) =\sum_{u\in V}f(u)$. The
minimum weight of a Roman dominating function on $G$ is called the
Roman domination number of $G$. In this article we study the Roman domination number of Cartesian  product graphs and strong product graphs. More precisely, we study the relationships between the Roman domination number of product graphs and the (Roman) domination number  of the factors.
\end{abstract}



\section{Introduction}

Nowadays the study of the behavior of several
graph parameters in product graphs have become an
interesting topic of research \cite{BookRall,BookImrich}. For instance, we emphasize the Shannon capacity of a graph \cite{shannon}, which is a certain
limiting value involving the vertex independence number of strong
product powers of a graph, and the Hedetniemi's coloring conjecture
for the categorical product \cite{hedetniemi,BookImrich}, which states that the chromatic number
of any categorial product graph is equal to the minimum value between
the chromatic numbers of its factors. Also, one of the oldest open
problems on domination in graphs is related with Cartesian product graphs. The
problem was presented first by Vizing  in 1963 \cite{vizing1}.  Vizing's conjecture
states that the domination number of any Cartesian product graph is
greater than or equal to the product of the domination numbers of its
factors.

Vizing's conjecture has become one of the most interesting problems on domination in graphs, which has led to develop some other kind of Vizing-like results for several parameters, even not related with standard domination. Many works have been developed in this sense and the conjecture has been proved for several families of graphs. The surveys \cite{bresar4,survey-vizing} contain almost all the results obtained around the conjecture. Also, in these surveys appear some references to similar open problems on product graphs. Nevertheless, the quantity of works about the conjecture have not been enough to finally prove or disprove it.
One variant of domination is the concept of Roman domination  introduced first by Steward
in \cite{roman-1} and studied further by other authors \cite{roman,PhDThesisDreyer,Favaron,M-Henning,roman-ciclos}. In this article we obtain Vizing-like results for the Roman domination number of Cartesian product graphs and strong product graphs.

We begin by establishing the principal terminology and notation
which we will use throughout the article. Hereafter $G=(V,E)$
denotes a finite simple graph. For two adjacent vertices $u$ and $v$ of $G$ we use the notation  $u\sim v$ and, in this case, we say that $uv$ is an edge
of $G$, i.e., $uv\in E$. For a  vertex
$v$ of $G$, $N(v)=\{u\in V:\; u\sim v\}$ denotes the set of neighbors that $v$ has in $G$. $N(v)$ is called the \emph{open neighborhood of} $v$
 and the \emph{close neighborhood of} $v$ is defined as $N[v]=N(v)\cup \{v\}$.
For a set $D\subseteq V$, the \emph{open neighborhood }is $N(D)=\cup_{v\in D}N(v)$ and the \emph{closed neighborhood} is $N[D]=N(D)\cup D$. A set  $D$ is a {\em dominating set} if $N[D]=V$.
 The \emph{domination number} $\gamma(G)$  is the minimum cardinality of a dominating set in $G$.  We say that a set $S$ is a $\gamma(G)$-set if it is a dominating set and $|S|=\gamma(G)$.

A map $f : V \rightarrow \{0, 1, 2\}$ is a \emph{ Roman dominating function } for a
graph $G$ if for every vertex $v$ with $f(v) = 0$, there exists a
vertex $u\in N(v)$ such that $f(u) = 2$. The {\em weight} of a Roman
dominating function is given by $f(V) =\sum_{u\in V}f(u)$. The
minimum weight of a Roman dominating function on $G$ is called the
{\em Roman domination number} of $G$ and it is denoted by
$\gamma_R(G)$.

Any Roman dominating function $f$ on a graph $G$ induces three sets $B_0,
B_1, B_2$, where $B_i = \{v\in V\;:\; f(v) = i\}$. Thus, we will write $f=(B_0,B_1,B_2)$.  It is clear that for any
Roman dominating function $f=(B_0,B_1,B_2)$ on a graph $G=(V,E)$ of order $n$ we
have that $f(V)=\sum_{u\in V}f(u)=2|B_2|+|B_1|$ and
$|B_0|+|B_1|+|B_2|=n$.  We say that a function $f=(B_0,B_1,B_2)$ is a $\gamma_{R}(G)$-function if it is a Roman domination function and $f(V)=\gamma_R(G)$.

Several results about the Roman dominating sets have been obtained in the last years, \cite{roman,PhDThesisDreyer,Favaron,M-Henning,roman-1,roman-ciclos}, and it is natural to try to relate the Roman domination number with the standard domination number. For instance, in \cite{roman,M-Henning} was obtained the following result, which we will use as a tool in this article.
\begin{lemma}{\em \cite{roman,M-Henning}}\label{lema-roman-dom}
For any graph $G$, $\gamma(G)\le \gamma_R(G)\le 2\gamma(G)$.
\end{lemma}

In this article we study the Roman domination number of Cartesian  product graphs and strong product graphs. More precisely, we study the relationships between the Roman domination number of product graphs and the domination number (Roman domination number) of the factors.

We recall that given two graphs $G$ and $H$ with set of vertices
$V_1=\{v_1,v_2,...,v_{n_1}\}$ and $V_2=\{u_1,u_2,...,u_{n_2}\}$,
respectively, the Cartesian product of $G$ and $H$ is
the graph $G\square H=(V,E)$, where $V=V_1\times V_2$ and  two vertices
$(v_i,u_j)$ and $(v_k,u_l)$ are adjacent in $G\square H$ if and only
if
\begin{itemize}
\item $v_i=v_k$ and $u_j\sim u_l$, or

\item $v_i\sim v_k$ and $u_j=u_l$.
\end{itemize}
The strong product $G\boxtimes H$ of the graphs $G$ and $H$ is defined on the Cartesian product of the vertex sets of the factors. Two distinct vertices  $(v_i,u_j)$ and $(v_k,u_l)$ of $G\boxtimes H$ being adjacent  with respect to the strong product if and only if
\begin{itemize}
 \item $v_i=v_k$ and $u_j\sim u_l$, or

 \item $v_i\sim v_k$ and $u_j=u_l$, or

\item  $v_i\sim v_k$ and $u_j\sim u_l$.
\end{itemize}
So, the Cartesian product graph $G\square H$ is a subgraph
of the strong product graph $G\boxtimes H$.

\section{Cartesian product graphs}

Currently there are few known results on the Roman domination number of Cartesian product graphs. As far as we know, the only works on this topic are as follows.
The Roman domination number of $C_{5t}\square C_{5k}$ was studied in \cite{roman-ciclos} and the Roman domination number of some grid graphs was studied in \cite{roman,PhDThesisDreyer}. Also, the following general relationship between the Roman domination number of Cartesian product graphs and the domination number of its factors was obtained in \cite{roman-arxiv}:
\begin{equation}\label{chino}
\gamma_R(G\Box H)\ge \gamma(G)\gamma(H).
\end{equation}

The following lemma will be helpful in obtaining the results reported here.

\begin{lemma}\label{remark-Rom-dom}
Let $G$ be a graph. For any  $\gamma_{R}(G)$-function $f=(B_0,B_1,B_2)$,
\begin{itemize}
\item[{\rm (i)}]  $|B_2|\le \gamma_R(G)-\gamma(G).$

\item[{\rm (ii)}] $|B_1|\ge 2\gamma(G)-\gamma_R(G).$
\end{itemize}
\end{lemma}

\begin{proof}
Since  $B_2\cup B_1$ is a dominating set for $G$ and $B_1\cap B_2=\emptyset$, we have
$\gamma(G)\le |B_2|+|B_1|$. So, (i) is deduced as $\gamma(G)=2|B_2|+|B_1|-|B_2|=\gamma_R(G)-|B_2|$, and (ii) is obtained as
$2\gamma(G)\le 2|B_2|+2|B_1|=2|B_2|+|B_1|+|B_1|=\gamma_R(G)+|B_1|.$
\end{proof}


\begin{theorem}\label{roman-cartesiano-inferior}
For any graphs $G$ and $H$,
\begin{itemize}
\item[{\rm (i)}] $\gamma_R(G\Box H)\ge \displaystyle\frac{2\gamma(G)\gamma_R(H)}{3}.$
\item[{\rm (ii)}] $\gamma_R(G\Box H)\ge \displaystyle\frac{\gamma(G)\gamma_R(H)+\gamma(G\Box H)}{2}.$
\end{itemize}
\end{theorem}

\begin{proof}
Let $V_1$ and $V_2$ be the vertex sets of $G$ and $H$, respectively.
Let $f=(B_0,B_1,B_2)$ be a $\gamma_{R}(G\Box H)$-function.
Let $S=\{u_1,u_2,...,u_{\gamma(G)}\}$ be a dominating set for $G$. Let $\{A_1,A_2,...,A_{\gamma(G)}\}$ be a vertex
partition of $G$ such that $u_i\in A_i$ and $ A_i\subseteq
N[u_i]$\footnote{Notice that this partition always exists, and it could be not unique.}.
Let $\{\Pi_1,\Pi_2,...,\Pi_{\gamma(G)}\}$ be a vertex partition of
$G\Box H$, such that $\Pi_i=A_i\times V_2$ for every
$i\in\{1,...,\gamma(G)\}$.

For every $i\in \{1,...,\gamma(G)\}$, let $f_i: V_2\rightarrow \{0,1,2\}$ be a function
such that
$f_i(v)=\max\{f(u,v)\;:\;u\in A_i\}$. For every $j\in \{0,1,2\}$, let $X_j^{(i)}=\{v\in V_2: \; f_i(v)=j\}$.  Now, let $Y_0^{(i)}\subseteq
X_0^{(i)}$ such that for every $v\in Y_0^{(i)}$, $N(v)\cap X_2^{(i)}=\emptyset$. Hence, we have that $f'_i=(X_0^{(i)}-Y_0^{(i)},X_1^{(i)}+Y_0^{(i)},X_2^{(i)})$ is  a Roman dominating function
on $H$. Thus,
\begin{align*}
\gamma_R(H)&\le 2|X_2^{(i)}|+|X_1^{(i)}|+|Y_0^{(i)}|\\
&\le 2|B_2\cap\Pi_i|+|B_1\cap\Pi_i| + |Y_0^{(i)}|.
\end{align*}
Hence,
\begin{align*}
\gamma_R(G\Box H)&=2|B_2|+|B_1|\\
&=\sum_{i=1}^{\gamma(G)}(2|B_2\cap\Pi_i|+|B_1\cap\Pi_i|)\\
&\ge\sum_{i=1}^{\gamma(G)}(\gamma_R(H)-|Y_0^{(i)}|)\\
&=\gamma(G)\gamma_R(H)-\sum_{i=1}^{\gamma(G)}|Y_0^{(i)}|.
\end{align*}
So,
\begin{equation}\label{cond-1}
\sum_{i=1}^{\gamma(G)}|Y_0^{(i)}| \ge \gamma(G)\gamma_R(H)-\gamma_R(G\Box H).
\end{equation}
Now, for every  $v\in V_2$, let $Z^v\in \{0,1\}^{\gamma(G)}$ be a binary  vector associated to $v$ as follows: $Z_i^v=1$ if $v\in Y_0^{(i)}$ and $Z_i^v=0$ if $v\not\in Y_0^{(i)}$.
So, $t_v=\| Z^v\|^2$ counts the number of components of $Z^v$ equal to one.  Hence,
\begin{equation}\label{sumasiguales}
\displaystyle\sum_{v\in
V_2}t_v=\displaystyle\sum_{i=1}^{\gamma(G)}|Y_0^{(i)}|.
\end{equation}

Notice that, if $Z^v_i=1$
and  $u\in A_{i}$, then vertex $(u,v)$ belongs to $B_0$. Moreover, $(u,v)$ is not adjacent to vertices of $B_2\cap \Pi_{i}$.
So, since $B_0$ is dominated by $B_2$, there exists $u'\in X_v=\{x\in V_1:\; (x,v)\in
B_2\}$ which is adjacent to $u$.  Hence,
$S_v=\left(S-\{u_i\in S:\; Z_i^v=1 \}\right)\cup X_v$
 is a dominating set for $G$.

Now, if $t_v >
|X_v|$, then we have
\begin{align*}
|S_v|&=|S|-t_v+|X_v|\\
&=\gamma(G)-t_v+|X_v|\\
&<\gamma(G)-t_v+t_v\\
&=\gamma(G),
\end{align*}
which is a contradiction. So, we have $t_v \le |X_v|$
and we obtain
\begin{equation}\label{cond-2}
\sum_{v\in V_2}t_v\le \sum_{v\in V_2}|X_v|=|B_2|,
\end{equation}
which leads to,
\begin{equation}
2\sum_{v\in V_2}t_v\le 2|B_2|+|B_1|=\gamma_R(G\Box H). \label{laotra}
\end{equation}
Thus, by  (\ref{cond-1}), (\ref{sumasiguales}) and (\ref{laotra}) we deduce
$$\gamma_R(G\Box H)\ge \gamma(G)\gamma_R(H)-\frac{\gamma_R(G\Box H)}{2},$$
and, as a consequence, (i) follows.

Now, By Lemma \ref{remark-Rom-dom} (i) and (\ref{cond-2}) we have
\begin{equation} \label{paraii}
\sum_{v\in V_2}t_v\le |B_2|\le \gamma_R(G\Box H)-\gamma(G\Box H).
\end{equation}
Thus,   by (\ref{cond-1}), (\ref{sumasiguales}) and (\ref{paraii}) we obtain (ii).
\end{proof}

Lemma \ref{lema-roman-dom} and Theorem \ref{roman-cartesiano-inferior} lead to the following result.
\begin{corollary} \label{Corollay1}
For any graphs $G$ and $H,$
\begin{itemize}
\item[{\rm (i)}] $\gamma_R(G\Box H)\ge \displaystyle\frac{\gamma_R(G)\gamma_R(H)}{3}.$
\item[{\rm (ii)}] $\gamma(G\Box H)\ge \displaystyle\frac{\gamma(G)\gamma_R(H)}{3}.$
\end{itemize}
\end{corollary}

Note that if there exists a graph that satisfies the above equalities, then Vizing's conjecture is false.

The following  inequality related to Vizing's conjecture was obtained in \cite{clark}:
\begin{equation}\label{CasiVizing}
\gamma(G\Box H)\ge \frac{\gamma(G)\gamma(H)}{2}.
\end{equation}
As the following Remark shows, if  $\gamma_R(H)>\frac{3\gamma(H)}{2}$, then  Corollary \ref{Corollay1} (ii) leads to a result which improves the above inequality.
\begin{remark}
Let $G$ and $H$ be two graphs. If  $\gamma_R(H)>\frac{3\gamma(H)}{2}$, then
$$\gamma(G\Box H)\ge \frac{\gamma(G)\gamma(H)}{2}+\frac{\gamma(G)}{3}.$$
\end{remark}
A graph $H$ is a \emph{Roman graph} if $\gamma_R(H)=2\gamma(H)$. Roman graphs were introduced in \cite{roman} where the authors presented some classes of Roman graphs and they proposed some open problems.
Theorem \ref{roman-cartesiano-inferior} (i) leads to the following result.
\begin{corollary}For any graph  $G$ and any Roman graph $H,$
\begin{itemize}
\item[{\rm (i)}] $\gamma_R(G\Box H)\ge \displaystyle\frac{4}{3}\gamma(G)\gamma(H).$
\item[{\rm (ii)}] $\gamma(G\Box H)\ge \displaystyle\frac{2}{3}\gamma(G)\gamma(H).$
\end{itemize}
\end{corollary}

Let $\mathfrak{F}$ be the class of all
graphs having a dominating set  $S=\{u_1,u_2,...,u_{{\gamma(G)}}\}$
such that $N[u_i]\cap N[u_j]=\emptyset$, for every $i,j\in
\{1,...,\gamma(G)\}$, $i\ne j$. In this case the set $S$ is called an \emph{efficient dominating set}.  Notice that  $\mathfrak{F}$ is the family of all  graphs having a perfect code\footnote{Given a graph $G=(V,E)$, a subset $S\subset V$ is a perfect code if $|N[v]\cap S|=1$, for every $v\in V$.}. Examples of graphs belonging to $\mathfrak{F}$ are the path graphs $P_n$, the cycle  graphs $C_{3k}$ and the cube graph $Q_3=K_2\Box K_2\Box K_2$.
Examples of Roman graphs belonging to  $\mathfrak{F}$ are $C_{3k}$, $P_{3k}$, $P_{3k+2}$ and $Q_3$. Note that $P_{3k+1}\in \mathfrak{F}$ but  $P_{3k+1}$ are not Roman paths, while $C_{3k+2}$ are Roman cycles but $C_{3k+2}\not\in \mathfrak{F}$.

 A $2$-packing of a graph $G$  is a set of vertices
in $G$ that are pair-wise at distance more than two. The
$2$-packing number $P_2(G)$ of a graph $G$ is the size of a largest $2$-packing in $G$.
The $2$-packing number is a graph invariant closely related to the domination number. In fact, it is well known that $P_2(G)\le \gamma(G)$, cf. \cite{BookRall,BookImrich}.

Let $G\in \mathfrak{F}$. Since every efficient dominating set  $S=\{u_1,u_2,...,u_{{\gamma(G)}}\}$ is a $2$-packing, we have $\gamma(G)\le P_2(G)$. So, we conclude that if $G\in \mathfrak{F}$, then $P_2(G) = \gamma(G)$ (The converse is not true). We recall that if $P_2(G) = \gamma(G)$, then Vizing's conjecture holds for $G$ \cite{BookImrich}. As a consequence, by Theorem \ref{roman-cartesiano-inferior} (ii) we deduce the following result which improves the inequality (\ref{chino}) when $G\in \mathfrak{F}$.
\begin{corollary}
Let $G$ and $H$ be two graphs.
 If $G\in \mathfrak{F}$,  then
$$\gamma_R(G\Box H)\ge \frac{1}{2}\max\left\{\gamma(G)\left(\gamma_R(H)+\gamma(H)\right), \gamma(H)\left(\gamma_R(G)+\gamma(G)\right)\right\}.$$
\end{corollary}


\begin{theorem}\label{roman-cartesiano-inferior1}
Let $G$ and $H$ be two graphs.
 If $G\in \mathfrak{F}$, then $$\gamma_R(G\Box H)\ge \gamma(G)\gamma_R(H).$$
\end{theorem}

\begin{proof}
Let $V_1$ and $V_2$ be the vertex sets of $G$ and $H$, respectively.
Let $S=\{u_1,u_2,...,u_{\gamma(G)}\}$ be an efficient dominating set for $G$, i.e.,
   $\{N[u_1],N[u_2],...,N[u_{\gamma(G)}]\}$ is a vertex
partition of $G$ and, as a consequence, 
$\{\Pi'_1,\Pi'_2,...,\Pi'_{\gamma(G)}\}$ is a vertex partition of
$G\Box H$, where $\Pi'_i=N[u_i]\times V_2$ for every
$i\in\{1,...,\gamma(G)\}$.

Proceeding analogously to the proof of Theorem \ref{roman-cartesiano-inferior}, we
consider a $\gamma_{R}(G\Box H)$-function  $f=(B_0,B_1,B_2)$ and, for every $i\in\{1,...,\gamma(G)\}$,
 we define the function $f_i:V_2\rightarrow \{0,1,2\}$ as
$f_i(v)=\max\{f(u,v)\;:\;u\in N[u_i]\}$. In addition, for every $j\in \{0,1,2\}$   we define $X_j^{(i)}=\{v\in V_2:f_i(v)=j\}$.

Now,  if  $v\in
X_0^{(i)}$, then for every $u\in N[u_i]$ we have that $(u,v)\in B_0$. Hence, since $u_i$ has no neighbors in $V_1-N[u_i]$ and $B_2$ dominates $B_0$, there exists $(u_i,v')\in B_2$ such that $v'$ is adjacent to $v$. We conclude that every $v\in
X_0^{(i)}$ has a neighbor  $v'\in X_2^{(i)}$ and, as a consequence,
$f_i=(X_0^{(i)},X_1^{(i)},X_2^{(i)})$ is a Roman dominating function on $H$, for every $i\in\{1,...,\gamma(G)\}$. Therefore,
the result is deduced as follows:
\begin{align*}
\gamma_R(G\Box H)&=2|B_2|+|B_1|\\
&=\sum_{i=1}^{\gamma(G)}\left(2|B_2\cap\Pi'_i|+|B_1\cap\Pi'_i|\right)\\
&\ge \sum_{i=1}^{\gamma(G)}\left( 2|X_2^{(i)}|+|X_1^{(i)}|\right)\\
&\ge \gamma(G)\gamma_R(H).
\end{align*}
\end{proof}

An interesting consequence of Theorem \ref{roman-cartesiano-inferior1} is the following result.
\begin{corollary}\label{F+Roman}
Let $G$ and $H$ be two graphs. If $G\in \mathfrak{F}$ and $H$ is a Roman graph, then $$\gamma_R(G\Box H)\ge 2\gamma(G)\gamma(H).$$
\end{corollary}





\begin{theorem}\label{superior}
For any graphs $G$ and $H$ of order $n_1$ and $n_2$, respectively,
 $$\gamma_R(G\Box H)\le \min\{n_1\gamma_R(H),n_2\gamma_R(G)\}.$$
\end{theorem}

\begin{proof}
 Let $f_1$ be a $\gamma_R(G)$-function. Let $f:V_1\times V_2\rightarrow \{0,1,2\}$ be a function  defined by $f(u,v)=f_1(u)$. If there exists a vertex $(x,y)\in V_1\times V_2$  such that $f(x,y)=0$, then $f_1(x)=0$. Since $f_1$ is Roman, there exists $u\in V_1$, adjacent to $x$, such that $f_1(u)=2$. Hence, we obtain that $f(u,y)=2$ and $(x,y)$ is adjacent to $(u,y)$. So, $f$ is a Roman dominating function. Therefore,
$$\gamma_R(G\square H)\le \sum_{(u,v)\in V_1\times V_2}f(u,v)=\sum_{v\in V_2}\sum_{u\in V_1} f_1(u)=\sum_{v\in V_2}\gamma_R(G)=n_2\gamma_R(G).$$
Analogously we obtain that $\gamma_R(G\square H)\le n_1\gamma_R(H)$ and the result follows.
\end{proof}

The above inequality is tight. It is achieved, for instance, for  $G=P_{n}$, a path graph of order $n$, and $H=S_{1,r}$, a star graph with $r\ge 2$ leaves. In this case  we have $\gamma_R(S_{1,r})=2=2\gamma(S_{1,r})$, $\gamma(P_{n})=\left\lceil\frac{n}{3}\right\rceil$,  $\gamma_R(P_{n})=\frac{2n+1}{3}$ if $n\equiv 1(3)$ and $\gamma_R(P_{n})=2\left\lceil\frac{n}{3}\right\rceil$ if $n\not\equiv 1(3)$. So,  $\gamma_R(G\square H)= 2n=n\gamma_R(H)$.

\begin{corollary}
For any graphs $G$ and $H$ of order $n_1$ and $n_2$, respectively,
 $$\gamma_R(G\Box H)\le 2\min\{n_1\gamma(H),n_2\gamma(G)\}.$$
\end{corollary}


\begin{lemma}{\em \cite{roman}}
\label{lemaRoman}
A graph $G$ is Roman if and only if it has a $\gamma_R(G)$-function $f=(A_0,A_1,A_2)$ with $|A_1|=0$.
\end{lemma}

\begin{theorem}\label{elde-k}
Let $G$ be a graph of order $n$ and let $H$ be a graph.
\begin{itemize}
\item[{\rm (i)}]
If $G$ has  at least one connected component of order greater than two, then
$$\gamma_R(G\Box H)\le (n+1)\gamma_R(H)-2\gamma(H).$$
\item[{\rm (ii)}]  If $G$ is a Roman graph, then
$$\gamma_R(G\Box H)\le 2n\left(\gamma_R(H)-\gamma(H)\right)+2\gamma(G)\left(2\gamma(H)-\gamma_R(H)\right).$$
\end{itemize}
\end{theorem}

\begin{proof}
Let $f_1=(A_0,A_1,A_2)$ be a $\gamma_R(G)$-function and let $f_2=(B_0,B_1,B_2)$ be $\gamma_R(H)$-function. We  define the map $f:V_1\times V_1\rightarrow \{0,1,2\}$ as follows.
\begin{itemize}
\item $f(u,v)=f_2(v)$ for every $(u,v)\notin (A_0\cup A_2)\times B_1$.
\item If $(u,v)\in A_0\times B_1$, then $f(u,v)=0$.
\item If $(u,v)\in A_2\times B_1$, then $f(u,v)=2$.
\end{itemize}
Since every vertex from $A_0\times B_1$ has a neighbor in $A_2\times B_1$ and every vertex of $V_1\times B_0$ has a neighbor in $V_1\times B_2$, we have that $f$ is a Roman dominating function on $G\square H$. Thus,
\begin{equation}\label{diferencia}
\gamma_R(G\square H)\le n\gamma_R(H)-|A_0||B_1|+|A_2||B_1|=n\gamma_R(H)-|B_1|(|A_0|-|A_2|).
\end{equation}
Since $G$ has at least one connected component of order greater than two, it is satisfied that $|A_0|\ge |A_2|+1$ and,  by Lemma \ref{remark-Rom-dom} (ii),
$|B_1|(|A_0|-|A_2|)\ge 2\gamma(H)- \gamma_R(H)$. Therefore,   by (\ref{diferencia}) we deduce (i).

Now, if $G$ is a Roman graph, then by Lemma \ref{lemaRoman} there exists a $\gamma_R(G)$-function $f=(A_0,A_1,A_2)$ with $|A_1|=0$. Thus, $|A_0|+|A_2|=n$ and, as a consequence,  $|A_0|-|A_2|=n-2\gamma(G)$. Therefore, by (\ref{diferencia}) we deduce (ii):
\begin{align*}
\gamma_R(G\square H)&\le n\gamma_R(H)-|B_1|(|A_0|-|A_2|)\\
                     &\le n\gamma_R(H)-\left(2\gamma(H)-\gamma_R(H)\right)\left(n-2\gamma(G)\right)\\
                     & =2n\left(\gamma_R(H)-\gamma(H)\right)+2\gamma(G)\left(2\gamma(H)-\gamma_R(H)\right).
\end{align*}
\end{proof}

For any Roman graph $H$, Theorem \ref{elde-k} leads to $\gamma_R(G\Box H)\le 2n\gamma(H)$. Now, for any non-Roman graph $H$ we have $\gamma_R(H)-2\gamma(H)\le -1$ and, as a consequence, Theorem \ref{elde-k} leads to the following result.
\begin{corollary}
Let $G$ be a graph of order $n$ and let $H$ be a graph. If $G$ has  at least one connected component of order greater than two and  $H$ is not Roman, then
$$\gamma_R(G\Box H)\le n\gamma_R(H)-1.$$
\end{corollary}

\begin{proposition}{\em \cite{roman}} \label{gamma+1}
If $G$ is a connected graph of order $n$, then $\gamma_R(G)=\gamma(G)+1$ if and only if there exists a vertex of $G$ of degree $n-\gamma(G)$.
\end{proposition}

From Proposition \ref{gamma+1} and Theorem \ref{elde-k} we derive the following  result.

\begin{proposition}\label{corochulo}
If $G$ is a graph of order $n_1$ having at least one connected component of order greater than two  and $H$ is  a connected graph of order $n_2$ having a vertex of degree $n_2-\gamma(H)$, then
$$\gamma_R(G\Box H)\le n_1(\gamma(H)+1)-\gamma(H)+1.$$
\end{proposition}

The above inequality is tight. For instance, if $G$ is a path graph of order three and $H$ is the star  $K_{1,3}$ with one of its edges subdivided, then we have  $\gamma(H)=2$ and $\gamma_R(G\Box H)=8$. So, Proposition \ref{corochulo} leads to the exact value of $\gamma_R(G\Box H)$.

\begin{theorem}\label{flojito}
For any graphs $G$ and $H$ of order $n_1$ and $n_2$, respectively,
$$\gamma_R(G\Box H)\le 2 \gamma(G)\gamma(H)+( n_1-\gamma(G))(n_2-\gamma(H)).$$
\end{theorem}

\begin{proof}
 Let $S_1$ be a $\gamma(G)$-set and let $S_2$ be a  $\gamma(H)$-set.  Let $B_2=S_1\times S_2$, $B_1=(V_1-S_1)\times (V_2-S_2)$ and $B_0=S_1\times (V_2-S_2)\cup (V_1-S_1)\times S_2$. Since $B_2$ dominates $B_0$, the map $f:V_1\times V_2\rightarrow \{0,1,2\}$ defined by
 $f(u,v)=i$, for every $(u,v)\in B_i$, is a Roman function on $G\Box H$. Therefore, the  result is obtained as follows,
 \begin{align*}
  \gamma_R(G\Box H)&\le 2|B_2|+|B_1|\\
  &=2|S_1||S_2|+|V_1-S_1||V_2-S_2|\\
  &=2\gamma(G)\gamma(H)+( n_1-\gamma(G))(n_2-\gamma(H)).
 \end{align*}
\end{proof}

We know that $\gamma_R(P_{3k+2})=2\gamma(P_{3k+2})=2(k+1)$, $\gamma_R(P_{3k+1})=2k+1$ and $\gamma(P_{3k+1})=k+1$. So,  Theorem \ref{flojito} leads to $\gamma_R \left( P_{3k+1} \square P_{3k+2}\right)\le 6k^2+6k+2$,  while by Theorem \ref{superior} we only get $\gamma_R \left(P_{3k+1} \square P_{3k+2}\right)\le 6k^2+7k+2$ and by Theorem \ref{elde-k} we only get $\gamma_R \left(P_{3k+2} \square P_{3k+1}\right)\le 6k^2+7k+1$.

From the above results we have that bounds on the Roman domination number  and the domination number of the factor graphs  lead to bounds on the Roman domination number of Cartesian product graphs. For example, it is well-known that  for any graph $G$ of order $n$ and maximum degree $\Delta$ it follows $\gamma(G)\ge \displaystyle\frac{n}{\Delta+1}$,  cf. \cite{bookdom1}.  The following straightforward result allow us to derive  several bounds on $\gamma_R(G\Box H)$.

\begin{remark}
For any graph $G\in \mathfrak{F}$  of order $n$ and minimum degree $\delta$,  $ \gamma(G)\le \frac{n}{\delta+1}.$ As a consequence, for any  $\delta$-regular graph   $G\in \mathfrak{F}$ it follows, $\gamma(G)=\frac{n}{\delta+1}.$
\end{remark}


An example of result derived from the above remark, Theorem \ref{roman-cartesiano-inferior1}  and Theorem \ref{superior}, is the following one.

\begin{proposition}
For any $\delta$-regular graph $G\in \mathfrak{F}$  of order $n$,
$$\frac{2n}{\delta+1}\le \gamma_R(G\Box K_2)\le \frac{4n}{\delta+1}.$$
\end{proposition}

\section{Strong product graphs}

 In this section we obtain some results on the  Roman domination number of strong product graphs. To begin with, we recall the following well-known result,  cf. \cite{BookImrich}.
\begin{theorem}{\em \cite{BookImrich}} \label{thetrivial} For any graphs $G$ and $H$,
$$\max\{P_2(G)\gamma(H),\gamma(G)P_2(H)\}\le \gamma(G\boxtimes H)\le \gamma(G)\gamma(H).$$
\end{theorem}

One immediate consequence of Theorem \ref{thetrivial} is the following result.
\begin{corollary}\label{igualdadDominationStrong}
For any graph $G\in \mathfrak{F}$ and any graph $H$, $\gamma(G\boxtimes H)=\gamma(G)\gamma(H).$
\end{corollary}

The next result follows from Lemma \ref{lema-roman-dom} and Theorem \ref{thetrivial}.
\begin{corollary} \label{coroloco} For any graphs $G$ and $H$,
$$\max\{P_2(G)\gamma(H),\gamma(G)P_2(H)\}\le \gamma_R(G\boxtimes H)\le 2\gamma(G)\gamma(H).$$
\end{corollary}


\begin{theorem}\label{roman-sup-completo-k}
Let  $f_1=(A_0,A_1,A_2)$ be a $\gamma_R(G)$-function and let
$f_2=(B_0,B_1,B_2)$ be a $\gamma_R(H)$-function. Then, $$\gamma_R(G\boxtimes H)\le
\gamma_R(G)\gamma_R(H)-2|A_2||B_2|.$$
\end{theorem}

\begin{proof}
 We define the function  $f$ on
$G\boxtimes H$ as follows:
$$f(u,v)=\left\{\begin{array}{cl}
                  2, & (u,v)\in (A_1\times B_2)\cup(A_2\times B_1)\cup (A_2\times B_2), \\
                  1, & (u,v)\in A_1\times B_1, \\
                  0, & \mbox{otherwise.}
                \end{array}
\right.$$
Note that the set $(A_0\times B_0)\cup (A_0\times B_2)\cup  (A_2\times B_0)$ is dominated by $A_2\times B_2$, the set $A_1\times B_0$ is dominated by $A_1\times B_2$, and $A_0\times B_1$ is dominated by $A_2\times B_1$. Then we have that $f$ is a  Roman dominating function on $G\boxtimes H$.





Therefore,
\begin{align*}
\gamma_R(G\boxtimes H)&\le 2|A_2||B_2|+2|A_1||B_2|+2|A_2||B_1|+|A_1||B_1|\\
&= 4|A_2||B_2|+2|A_1||B_2|+2|A_2||B_1|+|A_1||B_1|-2|A_2||B_2| \\
&=2|A_2|(2|B_2|+|B_1|)+|A_1|(2|B_2|+|B_1|)-2|A_2||B_2|\\
&=(2|A_2|+|A_1|)(2|B_2|+|B_1|)-2|A_2||B_2|\\
&=\gamma_R(G)\gamma_R(H)-2|A_2||B_2|.
\end{align*}

\end{proof}

Now we present some interesting consequences of Theorem \ref{roman-sup-completo-k}.
\begin{corollary}\label{roman-inf-sup}
For any non-empty graphs $G$ and $H$, $\gamma_R(G\boxtimes H)\le
\gamma_R(G)\gamma_R(H)-2.$
\end{corollary}

The above inequality is achieved,  for instance, if $G$ and $H$ are graphs of order $n_1$ and $n_2$, containing a vertex of degree $n_1-1$ and $n_2-1$, respectively.
In such a case, we have $\gamma_R(G\boxtimes H)\le \gamma_R(G)\gamma_R(H)-2=2\cdot 2-2=2.$

In order to show one example where Corollary \ref{roman-inf-sup} leads to better result than Corollary \ref{coroloco} we take a graph $G$ such that $\gamma_R(G)=\gamma(G)+1>3$ (see Proposition \ref{gamma+1}). In this case Corollary \ref{roman-inf-sup} leads to $\gamma_R(G\boxtimes G)\le (\gamma(G))^2+2\gamma(G)$, while Corollary \ref{coroloco} leads to $\gamma_R(G\boxtimes G)\le 2(\gamma(G))^2$.





If $H=P_n$ or $H=C_n$, then we have that for any  $\gamma_R(H)$-function  $f=(B_0,B_1,B_2)$, $|B_2|=\left\lfloor \frac{n}{3}\right\rfloor$. Hence, Theorem \ref{roman-sup-completo-k} leads to the following result.

\begin{corollary}
Let $G$ be a  non-empty graph. If $H=P_{n}$ or $H=C_{n}$, then
$$\gamma_R(G\boxtimes H)\le \left\{\begin{array}{ccc}
                                           \frac{2n+1}{3} \gamma_R(G)- 2\left\lfloor \frac{n}{3} \right\rfloor, & & n \equiv 1(3) \\
                                            &  &  \\
                                            2\left\lceil\frac{n}{3}\right\rceil \gamma_R(G)- 2\left\lfloor \frac{n}{3} \right\rfloor, & & n \not\equiv 1(3).
                                        \end{array} \right.
$$
\end{corollary}

 Every star graph $G=K_{1,r}$ satisfies the above equality for $n\not\equiv 2(3)$. In such a case we have $\gamma_R\left(C_{n}\boxtimes K_{1,r}\right)=\gamma_R\left(P_{n}\boxtimes K_{1,r}\right)=2\left\lceil \frac{n}{3} \right\rceil$. Note that   $C_{n}\boxtimes K_{1,r}$ and $P_{n}\boxtimes K_{1,r}$ are Roman graphs for $n\not\equiv 2(3)$.



\begin{theorem} \label{superiorstr}
Let $G$ and $H$ be two graphs. If $G\in \mathfrak{F}$, then $\gamma_R(G\boxtimes H)\ge \gamma(G)\gamma_R(H)$
\end{theorem}

\begin{proof}
Let $V_1$ and $V_2$ be the vertex sets of $G$ and $H$, respectively.
Let $S=\{u_1,u_2,...,u_{\gamma(G)}\}$ be an efficient dominating set for $G$, {\em i.e.}, $\{N_G[u_1],N_G[u_2],...,N_G[u_{\gamma(G)}]\}$ is a vertex
partition for $G$.
Let $\{\Pi_1,\Pi_2,...,\Pi_{\gamma(G)}\}$ be the vertex partition of
$G\boxtimes H$ defined as $\Pi_i=N_G[u_i]\times V_2$, for every
$i\in\{1,...,\gamma(G)\}$.

Now, let $f=(B_0,B_1,B_2)$ be a  $\gamma_R(G\boxtimes H)$-function and, for every $i\in \{1,...,\gamma(G)\}$, let the function $f^{(i)}:V_2\rightarrow \{0,1,2\}$ defined by $f^{(i)}(v)=\max\{f(u,v)\;:\;(u,v)\in \Pi_i\}$. Let $\{B_0^{(i)},B_1^{(i)},B_2^{(i)}\}$ such that $B_j^{(i)}=\{v\in V_2\;:\;f^{(i)}(v)=j\}$ with $j\in \{0,1,2\}$ and $i\in \{1,...,\gamma(G)\}$.

If there is a vertex $y$ of $H$ such that $f^{(i)}(y)=0$ and $N_{H}[y]\cap B_2^{(i)}=\emptyset$, then  $f(u_i,y)=0$ and $(u_i,y)$ is not adjacent to any
vertex $(a,b)$ of $G\boxtimes H$ with $f(a,b)=2$, a contradiction. Thus,
$f^{(i)}=(B_0^{(i)},B_1^{(i)},B_2^{(i)})$ is a Roman dominating function on $H$ for every $i\in \{1,...,\gamma(G)\}$. As a consequence,

\begin{align*}
\gamma_R(G\boxtimes H)&=2|B_2|+|B_1|\\
&=\sum_{i=1}^{\gamma(G)}(2|B_2\cap\Pi_i|+|B_1\cap\Pi_i|)\\
&\ge \sum_{i=1}^{\gamma(G)}(2|B_2^{(i)}|+|B_1^{(i)}|)\\
&\ge\sum_{i=1}^{\gamma(G)}\gamma_R(H)\\
&=\gamma(G)\gamma_R(H).
\end{align*}
Therefore, the proof is complete.
\end{proof}

Cockayne et al. \cite{roman} gave some classes of Roman graphs and they posed the following question: Can you find other classes of Roman graphs? The next result is an answer to this question.

\begin{theorem}\label{coroRomanStrProduct}
If $G\in \mathfrak{F}$ and $H$ is a Roman graph, then $G\boxtimes H$ is a Roman graph.
\end{theorem}

\begin{proof}
If  $G\in \mathfrak{F}$ and $H$ is  Roman, then Theorem \ref{superiorstr} leads to $\gamma_R(G\boxtimes H)\ge 2\gamma(G)\gamma(H)$. So, by Corollary \ref{coroloco} we obtain $\gamma_R(G\boxtimes H)=2\gamma(G)\gamma(H)$.  Hence, by Corollary \ref{igualdadDominationStrong} we conclude the proof.
\end{proof}

\end{document}